\documentclass[11pt,twoside]{amsart}

\usepackage{amsmath, amsthm, amscd, amsfonts, amssymb, graphicx, color}
\usepackage[bookmarksnumbered, plainpages]{hyperref}
\usepackage[capitalise]{cleveref}
\usepackage{comment}
\usepackage{calrsfs}
\usepackage{tikz-cd}
\usepackage{capt-of}

\usepackage[skins,theorems]{tcolorbox}
    \tcbset{highlight math style={size=fbox,
    colframe=red,colback=white,arc=2pt,boxrule=0.5pt}}

\usepackage{tikz}
\usetikzlibrary{arrows,decorations.markings,decorations.pathreplacing,calc,matrix,intersections}
\tikzset{->-/.style={decoration={markings,mark=at position #1 with {\arrow{>}}},postaction={decorate}}}

\newtheorem{thm}{Theorem}[section]
\newtheorem{cor}[thm]{Corollary}

\newtheorem*{thmA}{Theorem A}
\newtheorem*{thmB}{Theorem B}

\newcommand{\N}{\mathbb{N}}
\newcommand{\Z}{\mathbb{Z}}

\newcommand{\tM}{\mathbf{M}}
\newcommand{\tN}{\mathbf{N}}
\newcommand{\tS}{\mathbf{S}}
\newcommand{\tT}{\mathbf{T}}

\DeclareMathOperator{\height}{ht}

\numberwithin{equation}{section}

\def\pn{\par\noindent}

\begin{document}

\title[Conciseness on normal subgroups and new concise words from outer commutator words]{Conciseness on normal subgroups and new concise words from outer commutator words}

\author{Gustavo A.\ Fern\'andez-Alcober}
\address{{Department of Mathematics}, {University of the Basque Country UPV/EHU,} {Bilbao, Spain}}
\curraddr{}
\email{gustavo.fernandez@ehu.eus}
\thanks{}

\author{Matteo Pintonello}
\address{{Department of Mathematics}, {University of the Basque Country UPV/EHU,} {Bilbao, Spain}}
\curraddr{}
\email{matteo.pintonello@ehu.eus}
\thanks{}



\date{}

\dedicatory{}

\begin{abstract}
Let $w=w(x_1,\ldots,x_r)$ be an outer commutator word.
We show that the word $w(u_1,\ldots,u_r)$ is concise whenever $u_1,\ldots,u_r$ are non-commutator words in disjoint sets of variables.
This applies in particular to words of the form $w(x_1^{n_1},\ldots,x_r^{n_r})$, where the $n_i$ are non-zero integers.
Our approach is via the study of values of $w$ on normal subgroups, and in this setting we obtain the following result: if $N_1,\ldots,N_r$ are normal subgroups of a group $G$ and the set of all
values $w(g_1,\ldots,g_r)$ with $g_i\in N_i$ is finite then also the subgroup generated by these values, i.e.\ $w(N_1,\ldots,N_r)$, is finite.
\end{abstract}

\maketitle

\section{\bf Introduction}

\vskip 0.4 true cm

The goal of this article is to extend the main results of the paper \cite{Fe-Pi} from the case of lower central and derived words to the more general setting of outer commutator words.
More precisely, as was conjectured in \cite{Fe-Pi}, we have the following two results.

\begin{thmA}
Let $w=w(x_1,\ldots,x_r)$ be an outer commutator word.
If $u_1,\ldots,u_r$ are non-commutator words in disjoint sets of variables, then the word $w(u_1,\ldots,u_r)$ is concise.
In particular, the word $w(x_1^{n_1},\ldots,x_r^{n_r})$ is concise whenever $n_1,\ldots,n_r\in\Z\smallsetminus\{0\}$.
\end{thmA}

\begin{thmB}
Let $w=w(x_1,\ldots,x_r)$ be an outer commutator word.
Assume that $\tN=(N_1,\ldots,N_r)$ is a tuple of normal subgroups of a group $G$ such that
$w\{\tN\}$ is finite.
Then the subgroup $w(\tN)$ is also finite.
\end{thmB}

Here we keep the notation and terminology introduced in
\cite{Fe-Pi}.
Recall in particular that outer commutator words are defined recursively, starting from single variables, in such a way that if $\alpha$ and $\beta$ are outer commutators in
\emph{different variables}, then $w=[\alpha,\beta]$ is also an outer commutator.
Also, non-commutator words are the words lying outside the derived subgroup of the free group of infinite rank, i.e.\ whose exponent sum in at least one of the variables is not zero.

Let $w$ be a word.
If we write $w=w(x_1,\ldots,x_r)$, we tacitly assume that $w$ uses \emph{all} variables $x_1,\ldots,x_r$.
Given a tuple $\tS=(S_1,\ldots,S_r)$ of subsets of a group $G$, we let
$w\{\tS\}$ be the set of all values $w(g_1,\ldots,g_r)$ with
$g_i\in S_i$ for every $i=1,\ldots,r$, and $w(\tS)$ is the subgroup generated by $w\{\tS\}$.
If all subsets $S_i$ are equal to $G$, we simply write $w\{G\}$ and $w(G)$, respectively.
The word $w$ is \emph{concise} if whenever the set $w\{G\}$ is finite in a group $G$, then also the subgroup $w(G)$ is finite.

For the proofs of Theorems A and B, we rely on the same ideas that we used for the corresponding results in \cite{Fe-Pi}.
The key is the following theorem on the existence of a linear series for subgroups of the form $w(\tN)$, where $w$ is an outer commutator word.
Following \cite{Fe-Mo}, we define the \emph{height} $\height(w)$ of an outer commutator word $w$ inductively, with a single variable having height 0, and with the height of $w=[\alpha,\beta]$ being
\begin{equation}
\label{def height}
\height(w) = 1 + \max \{ \height(\alpha), \height(\beta) \}.
\end{equation}
Notice that the height of an outer commutator in $r$ variables will always be at most $r-1$.

For the concepts of extended word, outer commutator extension of a tuple of normal subgroups, and extension by outer commutators, which appear in the statement of the following theorem, we refer the reader to Definitions 3.1, 3.3, and 3.4 of \cite{Fe-Pi}.

\begin{thm}
\label{linear series ocw}
Let $w$ be an outer commutator word in $r$ variables, of height $h$, and let $\tN=(N_1,\ldots,N_{r})$ be a tuple of normal subgroups of a group $G$.
Then there exists a series
\begin{equation}
\label{linear series w(N)}
[w(\tN),w(\tN)] = V_0\le V_1\le \cdots \le V_t = w(\tN)
\end{equation}
of normal subgroups of $G$ such that, for every fixed value $i=1,\ldots,t$, the following hold:
\begin{enumerate}
\item
The section $V_i/V_{i-1}$ is the image of an extension $v_i(\tM_i)$ of $w(\tN)$ by outer commutators.
\item
The extended word $v_i$ depends only on $w$, and not on the group $G$ or on the tuple $\tN$. Furthermore, $v_i$ has degree either $0$, if $h=0$, or at most $h-1$ otherwise.
\item
The procedure of construction of the outer commutator extension $\tM_i$ from the tuple $\tN$ depends only on $w$, and not on the group $G$ or on the tuple $\tN$.
\item
In the section $V_i/V_{i-1}$, the word $v_i$ is linear in one component of the tuple $\tM_i$.
\end{enumerate}
Furthermore, the length $t$ of the series is 1 if $h=0$ and at most $2^h+2^{h-1}-1$ if $h\geq 1$.
\end{thm}

\begin{proof}
We prove the theorem by induction on $h$.
If $h=0$ then $w$ is a single variable, and the result is obvious.
On the other hand, if $h=1$ then we may assume that $w=[x_1,x_2]$, and the theorem holds by Theorem 2.9 of \cite{Fe-Pi}.
More precisely, we take the series
\[
V_0 = \big[[N_1,N_2],[N_1,N_2]\big] \le V_1 = \big[N_1,[N_1,N_2]\big] \le V_2= [N_1,N_2],
\]
and $v_1=v_2=w$, $\tM_1=(N_1,[N_1,N_2])$ and $\tM_2=(N_1,N_2)$.

Now let $h\ge 2$ and let us write $w=[\alpha,\beta]$, where $\alpha$ and $\beta$ are outer commutator words.
After renaming the variables involved, we may assume that
$\alpha=\alpha(x_1,\ldots,x_q)$ and $\beta=\beta(x_{q+1},\ldots,x_r)$.
By \eqref{def height}, we have $\height(\alpha),\height(\beta)\le h-1$ and the induction hypothesis applies to both
$\alpha$ and $\beta$.
Thus if we set $\tN^{\alpha}=(N_1,\ldots,N_{q})$ and
$\tN^{\beta}=(N_{q+1},\ldots,N_r)$ then there exist two series
\begin{equation}
\label{series(k-1)-1 ocw}
A_{0}=[\alpha(\tN^{\alpha}),\alpha(\tN^{\alpha})]
\le
\cdots
\le
A_{i}
\le
\cdots
\le
A_{t^{\alpha}} =\alpha(\tN^{\alpha})
\end{equation}
and
\begin{equation}
\label{series(k-1)-2 ocw}
B_{0}=[\beta(\tN^{\beta}),\beta(\tN^{\beta})]
\le
\cdots
\le
B_{i}
\le
\cdots
\le
B_{t^{\beta}} = \beta(\tN^{\beta})
\end{equation}
such that the corresponding factors $A_{i}/A_{i-1}$ and $B_{i}/B_{i-1}$ are the images of 
$v_i^{\alpha}(\tM^\alpha_{i})$ and $v_i^{\beta}(\tM^\beta_{i})$, respectively, where:
\begin{enumerate}
\item[(a)]
$v_i^{\alpha}$ and $v_i^{\beta}$ are extended words of $\alpha$ and $\beta$, respectively, each of degree at most $h-2$.
\item[(b)]
$\tM^\alpha_{i}$ and $\tM^\beta_{i}$ are outer commutator extensions of $\tN^{\alpha}$ and $\tN^{\beta}$, respectively.
\item[(c)]
The extended words $v_i^{\alpha}$ and $v_i^{\beta}$, and the procedure of construction of the outer commutator extensions $\tM_i^{\alpha}$
and $\tM_i^{\beta}$, depend only on $\alpha$ and $\beta$, respectively, and not on the group $G$ or on the tuples $\tN^{\alpha}$ and $\tN^{\beta}$.
\item[(d)]
In the sections $A_{i}/A_{i-1}$ and $B_{i}/B_{i-1}$, the words $v_i^{\alpha}$ and $v_i^{\beta}$ are linear in one component
of the tuples $\tM_i^{\alpha}$ and $\tM_i^{\beta}$, respectively.
\end{enumerate}
Also, since $h\geq 2$, we have $t^{\alpha},t^{\beta}\le 2^{h-1}+2^{h-2}-1$.

Let us now see how to construct the series for $w$ and for the tuple $\tN$ from the two series \eqref{series(k-1)-1 ocw} and \eqref{series(k-1)-2 ocw}.
As we will see, the length $t$ of this new series will be
\[
t = t^{\alpha}+t^{\beta}+1 \le 2^h+2^{h-1}-1,
\]
as desired.

We start by taking the commutator of all terms of the series \eqref{series(k-1)-1 ocw} with $\beta(\tN^{\beta})$.
This way we obtain the series
\begin{equation}
\label{series(k-1)-1-comm ocw}
[A_{0},\beta(\tN^{\beta})]
\le
\cdots
\le
[A_{i},\beta(\tN^{\beta})]
\le
\cdots
\le
[\alpha(\tN^{\alpha}),\beta(\tN^{\beta})]
=
w(\tN).
\end{equation}
By P.\ Hall's Three Subgroup Lemma, we have
\begin{equation*}
\begin{split}
[A_{0},\beta(\tN^{\beta})]
&=
[\alpha(\tN^{\alpha}),\alpha(\tN^{\alpha}),\beta(\tN^{\beta})]
\\
&\le
[\alpha(\tN^{\alpha}),\beta(\tN^{\beta}),\alpha(\tN^{\alpha})]
= 
[\alpha(\tN^{\alpha}),w(\tN)].
\end{split}
\end{equation*}
Now we multiply all terms of the series \eqref{series(k-1)-1-comm ocw} by
$[\alpha(\tN^{\alpha}),w(\tN)]$, and this is the rightmost part of the series we are seeking (where $t=t^{\alpha}+t^{\beta}+1$, as above):
\begin{equation}
\label{seriesk-1 ocw}
V_{t^{\beta}+1}
=
[\alpha(\tN^{\alpha}),w(\tN)]
\le
\cdots
\le
V_{t^{\beta}+i+1}
=
[A_{i},\beta(\tN^{\beta})] \, [\alpha(\tN^{\alpha}),w(\tN)]
\le
\cdots
\le
V_t = w(\tN).
\end{equation}
For every $i=1,\ldots,t^{\alpha}$, the factor $V_{t^{\beta}+i+1}/V_{t^{\beta}+i}$ in this series is the image of the subgroup
\[
[v_i^\alpha(\tM^\alpha_{i}),\beta(\tN^{\beta})],
\]
which can be represented in the form
$v_{t^{\beta}+i+1}(\tM_{t^{\beta}+i+1})$ by taking
\[
v_{t^{\beta}+i+1} = [v_i^{\alpha},\beta(x_{q+1},\ldots,x_r)]
\]
and defining $\tM_{t^{\beta}+i+1}$ to be the concatenation of $\tM_i^{\alpha}$ and $\tN^{\beta}$, where the elements of $\tN^{\beta}$ occupy the positions
corresponding to the variables $x_{q+1},\ldots,x_r$.
Notice that $v_{t^{\beta}+i+1}$ is an extended word of $w$ of the same degree as $v_i^{\alpha}$, so at most $h-2$, since
$\beta(x_{q+1},\ldots,x_r)$ does not involve any variables from the set $Y$.
Furthermore, $v_{t^{\beta}+i+1}$ only depends on $w$, and
$\tM_{t^{\beta}+i+1}$ is an outer commutator extension of $\tN$ whose construction only depends on $w$.
Of course, $v_{t^{\beta}+i+1}$ and $\tM_{t^{\beta}+i+1}$ will also depend on $i$, but we are working with a fixed value of $i$. 

In a symmetric way, by first taking the commutator of $\alpha(\tN^{\alpha})$ with all terms of the series \eqref{series(k-1)-2 ocw} and then multiplying by
$[w(\tN),\beta(\tN^{\beta})]$, we get the series
\begin{equation}
\label{series-k-2 aux ocw}
U_{t^{\alpha}+1}
=
[w(\tN),\beta(\tN^{\beta})]
\le
\cdots
\le
U_{t^{\alpha}+i+1}
=
[\alpha(\tN^{\alpha}),B_{i}] \, [w(\tN),\beta(\tN^{\beta})]
\le
\cdots
\le
U_t = w(\tN).
\end{equation}
Following the notation in Definition 3.1 of \cite{Fe-Pi}, suppose that the variables from the set $Y$ that have been used to construct
the extended words $v_i^{\beta}$ are contained in the set $\{y_1,\ldots,y_s\}$.
Then the factor $U_{t^\alpha+i+1}/U_{t^\alpha+i}$ of this last series can be generated by the image of the subgroup
$u_{t^{\alpha}+i+1}(\mathbf{L}_{t^{\alpha}+i+1})$, where
\[
u_{t^{\alpha}+i+1} = [\alpha(y_{s+1},\ldots,y_{s+q}),v_i^{\beta}]
\]
and $\mathbf{L}_{t^{\alpha}+i+1}$ is the concatenation of $\tM_i^{\beta}$ and
$\tN^{\alpha}$, where the elements of $\tN^{\alpha}$ occupy the positions
corresponding to the $q$ variables $y_{s+1},\ldots,y_{s+q}$.
Note that $u_{t^{\alpha}+i+1}$ is an extended word of $\beta(x_{q+1},\ldots,x_r)$ of degree at most $h-1$ that only depends on $w$.

Now we take the commutator of $\alpha(\tN^{\alpha})$ with the terms of the last series, and subtract $t^{\alpha}$ to all indices, getting
\begin{equation}
\label{seriesk-2 aux-2 ocw}
Z_1
=
\big[\alpha(\tN^{\alpha}),[w(\tN),\beta(\tN^{\beta})]\big]
\le
\cdots
\le
Z_i
=
[\alpha(\tN^{\alpha}),U_{t^{\alpha}+i}]
\le
\cdots
\le
Z_{t^{\beta}+1} = [\alpha(\tN^{\alpha}),w(\tN)].
\end{equation}
Finally, we define $V_0=[w(\tN),w(\tN)]$ and multiply all terms of \eqref{seriesk-2 aux-2 ocw} by this subgroup, setting
$V_i=Z_iV_0$ for $i=1,\ldots,t^{\beta}+1$.

Since $V_0\le [\alpha(\tN^{\alpha}),w(\tN)]$, we get the series
\begin{equation}
\label{seriesk-2 ocw}
V_0 = [w(\tN),w(\tN)]
\le
\cdots
\le
V_i
=
Z_i [w(\tN),w(\tN)]
\le
\cdots
\le
V_{t^{\beta}+1} = [\alpha(\tN^{\alpha}),w(\tN)].
\end{equation}
In this series, the factors $V_i/V_{i-1}$ for $i=2,\ldots,t^{\beta}+1$ are given by the images of the subgroups $v_i(\mathbf{M}_i)$, where
\[
v_i = [\alpha(x_1,\ldots,x_q),u_{t^{\alpha}+i}]
\]
is an extended word of $w$ of degree at most $h-1$, and $\tM_i$ is the concatenation of $\tN^{\alpha}$ and $\mathbf{L}_{t^{\alpha}+i}$, where
the elements of $\tN^{\alpha}$ occupy the positions corresponding to the variables $x_1,\ldots,x_q$.
On the other hand, the quotient $V_1/V_0$ is given by the image of $v_1(\tM_1)$, where
\[
v_1 = [\alpha(x_1,\ldots,x_{q}),[y_1,\beta(x_{q+1},\ldots,x_r)]]
\]
is an extended word of $w$ of degree $1$, and $\tM_1=(\tN,w(\tN))$, with $w(\tN)$ corresponding to $y_1$.

Now the concatenation of \eqref{seriesk-1 ocw} and \eqref{seriesk-2 ocw} is the desired series for $w$ and $\tN$.
The discussion of the previous paragraphs shows that $v_i(\tM_i)$ is an extension of $w(\tN)$ for every $i=1,\ldots,t$.
Thus we only need to check linearity of every word $v_i$ in one component of the vector $\tM_i$.
For $i=t^{\beta}+1,\ldots,t$, if $v_i^\alpha$ is linear in a component of $\tM^\alpha_{i}$ in the section $V_i^{\alpha}/V_{i-1}^{\alpha}$,
then $v_i$ is linear in the same component of $\tM_i$ in the section $V_i/V_{i-1}$, by Lemma 2.8 of \cite{Fe-Pi}.
For $i=1,\ldots,t^{\beta}$, we can use similarly the linearity in the section $V_i^{\beta}/V_{i-1}^{\beta}$. 
Finally for $i=1$, since $V_0 = [w(\tN),w(\tN)]$ we have linearity in the component corresponding to $y_1$, which takes values in $w(\tN)$.
\end{proof}

Note that Theorem 1.1 generalizes Theorem 3.5 of \cite{Fe-Pi}, where we provided a series like \eqref{linear series w(N)} for every derived word
$\delta_k$, of length exactly $2^k+2^{k-1}-1$.
Since $\delta_k$ is of height $k$, the bound for the length of the series given in Theorem 1.1 is attained for derived words.

\vspace{10pt}

We can now prove the corresponding version of Theorems 2.11 and 3.6 of \cite{Fe-Pi} for an arbitrary outer commutator word $w$.

\begin{thm}
\label{ocw concise most general}
Let $w$ be an outer commutator word in $r$ variables and let
$\tN=(N_1,\ldots,N_{r})$ be a tuple of normal subgroups of a group $G$.
Assume that $N_i=\langle S_i \rangle$ for every $i=1,\ldots,r$, where:
\begin{enumerate}
\item
$S_i$ is a normal subset of $G$.
\item
There exists $n_i\in\N$ such that all $n_i$th powers of elements of $N_i$ are contained in $S_i$.
\end{enumerate}
If for the tuple $\tS=(S_1,\ldots,S_{r})$ the set of values $w\{\tS\}$ is finite of order $m$, then the subgroup $w(\tN)$ is also finite and
of $(m,r,n_1,\ldots,n_{r})$-bounded order.
\end{thm}

\begin{proof}
Let us consider the series
\[
[w(\tN),w(\tN)] = V_0\le V_1\le \cdots
\le V_t = w(\tN)
\]
of \cref{linear series ocw}.
We have proved that the length $t$ of this series is either $1$ or less than equal to $2^h+2^{h-1}-1$, where $h$ is the height of the outer commutator word $w$.
So in any case we have $t\le 2^r-1$, since $h\le r-1$.
We prove that every $V_i$ is finite of bounded order by induction on $i$.
The result for $i=0$ follows from Lemma 2.10 of \cite{Fe-Pi}.
Assume now that $i\ge 1$ and that the result is true for $V_{i-1}$.
By \cref{linear series ocw}, the section $V_i/V_{i-1}$ coincides with the image of a subgroup $v_i(\tM_i)$ that is an extension of $w(\tN)$ of degree at most $h$ in any case.

Let $\tM_i=(M_1,\ldots,M_s)$, which is an outer commutator extension of $\tN$.
Hence $s\ge r$ and for every $j=1,\ldots,s$ we have $M_j=w_j(\tN_j)$, where $w_j$ is an outer commutator word, all components in $\tN_j$ belong to $\tN$, and one of these components must be $N_j$ for $j=1,\ldots,r$.

Let $T_j=w_j\{\tS_j\}$, where $\tS_j$ is obtained from $\tN_j$ by replacing each subgroup $N_{\ell}$ with its given generating set $S_{\ell}$.
Hence $T_j$ is a normal subset of $G$ that generates $M_j$.
Also, for a given $j\in\{1,\ldots,r\}$, the subset $S_j$ appears as a component of $\tS_j$ and, since $\tS_j$ consists of normal subsets of $G$
by condition (i), it follows from Lemma 2.5 of \cite{Fe-Pi} that $T_j\subseteq S_j^{*q}$, where $q$ is the number
of variables used by $w_j$.
Now recall from \cref{linear series ocw} that the word $v_i$ (and hence also the number $s$ of variables of $v_i$) and the words $w_1,\ldots,w_s$ only depend on $w$, and not on $G$ or on $\tN$.
Since there are only finitely many outer commutator words in $r$ variables, there is a function
$f:\N\rightarrow\N$ such that $T_j\subseteq S_j^{\ast f(r)}$ for every $j=1,\ldots,r$.
If we set $\tT=(T_1,\ldots,T_s)$ then, by Lemma 3.2 of \cite{Fe-Pi}, we get $v_i\{\tT\}\subseteq w\{\tS\}^{\ast n}$, where
\[
n = f(r)^{r} 2^{h}\leq f(r)^{r} 2^{r-1}.
\]
Consequently
\begin{equation}
\label{bound v(T)}
|v_i\{\tT\}| \le (2m+1)^n,
\end{equation}
and $v_i\{\tT\}$ is finite of $(m,r)$-bounded cardinality.
On the other hand, it follows from Lemma 2.1 of \cite{Fe-Pi} that $v_i(\tM_i)$ can be generated by the set of values $v_i\{\tT\}$.

From \cref{linear series ocw}, we know that the word $v_i$ is linear in some position $j\in\{1,\ldots,s\}$ of the tuple $\tM_i$ modulo $V_{i-1}$.
Since $M_j=w_j(\tN_j)$ is as above, we have $M_j\le N_{\ell}$ for some $\ell\in\{1,\ldots,r\}$, and actually $\ell=j$
if $j\in\{1,\ldots,r\}$.
Now, from linearity, for every tuple $\mathbf{t}=(t_1,\ldots,t_s)\in \tT$ and every $\lambda\in\Z$, we have
\begin{equation}
\label{power out of v}
v_i(\mathbf{t})^{\lambda n_{\ell}}
=
v_i(t_1,\ldots,t_j,\ldots,t_s)^{\lambda n_{\ell}}
\equiv
v_i(t_1,\ldots,t_j^{\lambda n_{\ell}},\ldots,t_s)
\pmod{V_{i-1}}.
\end{equation}
We have $t_j^{\lambda}\in M_j\le N_{\ell}$ and then, by (ii) in the statement of the theorem, $t_j^{\lambda n_{\ell}}\in S_{\ell}$.
So we get
\[
v_i(t_1,\ldots,t_j^{\lambda n_{\ell}},\ldots,t_s) \in v_i\{\tT_j\},
\]
where $\tT_j$ is the tuple obtained from $\tT$ after replacing $T_j$ with $S_{\ell}$.
Similarly to \eqref{bound v(T)} and taking into account that $\ell=j$ if $j\in\{1,\ldots,r\}$, it follows that the set $v_i\{\tT_j\}$ is finite of $(m,r)$-bounded cardinality.
By \eqref{power out of v} there exist $(m,r)$-bounded integers $\lambda$ and $\mu$, with $\lambda\ne\mu$, such that
\[
v_i(\mathbf{t})^{\lambda n_{\ell}}
\equiv
v_i(\mathbf{t})^{\mu n_{\ell}}
\pmod{V_{i-1}}.
\]
This implies that $v_i(\mathbf{t})$ has finite order modulo $V_{i-1}$, bounded in terms of $m$, $r$ and $n_{\ell}$.

Summarizing, the abelian quotient $V_i/V_{i-1}$ is the image of the verbal subgroup $v_i(\tM_i)$, which is generated by the set $v_i\{\tT\}$
of $(m,r)$-bounded cardinality, and each element of $v_i\{\tT\}$ has $(m,r,n_{\ell})$-bounded order.
We conclude that the order of $V_i/V_{i-1}$ is $(m,r,n_{\ell})$-bounded, which completes the proof.
\end{proof}

We can now obtain Theorems A and B as special cases of this last result.

\begin{cor}
\label{ocw of non-comm concise}
Let $w$ be an outer commutator word in $r$ variables and let $u_1,\ldots,u_{r}$ be non-commutator words.
Then the word $w(u_1,\ldots,u_{r})$ is boundedly concise.
In particular, $w(x_1^{n_1},\ldots,x_{r}^{n_{r}})$ is boundedly concise for all $n_i\in \Z\smallsetminus\{0\}$.
\end{cor}

\begin{cor}
\label{ocw concise on normal}
Let $w$ be an outer commutator word in $r$ variables and let $\tN=(N_1,\ldots,N_{r})$ be a tuple of normal subgroups of a group $G$.
If $w\{\tN\}$ is finite of order $m$, then the subgroup $w(\tN)$ is also finite, of $(m,r)$-bounded order.
\end{cor}


\vskip 0.4 true cm

\begin{center}{\textbf{Acknowledgments}}
\end{center}
Supported by the Spanish Government, grant PID2020-117281GB-I00, partly with FEDER funds, and by the Basque Government, grant IT483-22.
The second author was also supported by a grant FPI-2018 of the Spanish Government.

The authors thank the anonymous referee for helpful comments. 

\vskip 0.4 true cm



\bigskip
\bigskip

{\footnotesize \pn{\bf Gustavo A.\ Fern\'andez-Alcober}\;
\\
{Department of Mathematics}, {University of the Basque Country UPV/EHU,} {Bilbao, Spain}\\
{\tt Email: gustavo.fernandez@ehu.eus}}\\

{\footnotesize \pn{\bf Matteo Pintonello}\;
\\
{Department of Mathematics}, {University of the Basque Country UPV/EHU,} {Bilbao, Spain}\\
{\tt Email: matteo.pintonello@ehu.eus}}\\

\end{document}